\documentclass[a4paper,11pt]{article}
\usepackage{amssymb}
\usepackage{amsfonts}
\usepackage{amsmath}
\usepackage{amsrefs}
\usepackage{epsfig}
\usepackage{amsthm}
\usepackage{color}
\usepackage[]{epsfig}
\usepackage[latin1]{inputenc}

\newtheorem{theorem}{Theorem}[section]
\newtheorem{proposition}{Proposition}[section]
\newtheorem{lemma}{Lemma}[section]
\newtheorem{definition}{Definition}[section]

\newtheorem{remark}{Remark}[section]
\newtheorem{claim}{Claim}%[section]
\newcommand{\R}{\mathbb{R}}
\newcommand{\h}{\mathbb{H}}
\newcommand{\s}{\mathbb{S}}

\newcommand{\al}{\alpha}
\newcommand{\ria}{\rightarrow}

\newcommand{\n}{\nabla}

\newcommand{\ran}{\rangle}
\newcommand{\lan}{\langle}

\DeclareMathOperator{\hess}{Hess}
\DeclareMathOperator{\ric}{Ric}

\DeclareMathOperator{\di}{div}
\DeclareMathOperator{\tr}{tr}

\DeclareMathOperator{\Rm}{ \textrm{Rm}}
\numberwithin{equation}{section}
\begin{document}
\title{Eigenvalue estimates for a class of elliptic differential operators on compact manifolds}
\author{Hil\'ario Alencar\footnote{Hil\'ario Alencar was partially supported by CNPq and Fapeal of Brazil.}, Greg\'orio Silva Neto, Detang Zhou \footnote{Detang Zhou was partially supported by CNPq and Faperj of Brazil.}}
%\date{September 30, 2014}
\date{March 24, 2015}
%\email{gregoriosilvaneto@gmail.com}
%%%%%%%%%%%%%%%%%%%%%%%%%%%%%%%%%%%%%%%%%%%%%%%%%%%%%%%%%%%%%%%%%%%%%%%

%%%%%%%%%%%%%%%%%%%%%%%%%%%%%%%%%%%%%%%%%%%%%%%%%%%%%%%%%%%%%%%%%%%%%%%

\maketitle

%\title{Eigenvalue estimates for elliptic differential operators on compact manifolds}
%Grants or other notes



%\author{Hil\'ario Alencar \  Greg\'orio Silva Neto \ Detang Zhou}


%\subjclass[2000]{MSC 53C42 }

\maketitle
\baselineskip=16pt

\begin{abstract}
The motivation of this paper is to study a second order elliptic operator which appears naturally in Riemannian geometry, for instance in the study of hypersurfaces with constant $r$-mean curvature. We  prove a  generalized Bochner-type formula for such a kind of operators and as applications we obtain some sharp estimates for the first  nonzero eigenvalues in two special cases. These results can be considered as generalizations of  the Lichnerowicz-Obata Theorem.

%\end{abstract}
%Insert your abstract here. Include keywords, PACS and mathematical
%subject classification numbers as needed.
%\PACS{PACS code1 \and PACS code2 \and more}
\end{abstract}
\section{Introduction}

Let $\{\omega_1,\ldots,\omega_n\}$ be a local coframe field defined on a Riemannian manifold $(M,g).$  For  a symmetric tensor 
$\phi=\sum_{i,j=1}^n\phi_{ij} \omega_i\otimes\omega_j$   on $M,$
 Cheng and Yau, see \cite{yau}, define an operator associated to $\phi$ by
\begin{equation}\label{defi.square}
\square f = \sum_{i,j=1}^n\phi_{ij}f_{ij}.
\end{equation}

In this paper, we prove  the following  new Bochner type formula.
\begin{theorem}\label{teo.boch-intro} Let  $M^n$ be a Riemannian manifold and  $\displaystyle{\phi=\sum_{i,j=1}^n\phi_{ij}\omega_i\otimes\omega_j}$ be a symmetric tensor defined on $M$. Then, for any smooth function $f:M\ria\R$, and for any $c\in\R,$
\begin{equation}\label{theo.bochner}
\begin{split}
\dfrac{1}{2}\square(|\n f|^2)&=\lan\n f,\n(\square f)\ran + \lan \phi (\n f), \n(\Delta f)\ran + 2\sum_{i,j,k=1}^n\phi_{ij}f_{jk}f_{ki}+ 2\sum_{i,j,k,m=1}^nf_if_j\phi_{im}R_{mkjk}\\
&\qquad  + c\sum_{i,j=1}^n(\tr\phi)_{ij}f_if_j - \sum_{i,j=1}^nf_if_j\Delta\phi_{ij} + \sum_{i,j=1}^nf_if_j\left(\sum_{k=1}^n\phi_{ikk} - c\sum_{k=1}^n\phi_{kki}\right)_j\\
&\qquad +\sum_{k=1}^n\left(\sum_{i,j=1}^n f_if_j(\phi_{jik} - \phi_{jki})\right)_k - \sum_{k=1}^n\left(\sum_{i,j=1}^nf_j\phi_{ij}f_{ik}\right)_k.\\
\end{split}
\end{equation}
\end{theorem}
%We would like to present some remarks before the applications.
\begin{remark}
If $\phi $ is equal to the metric $g$,    then    $ \sum_{k=1}^n\left(\sum_{i,j=1}^nf_j\phi_{ij}f_{ik}\right)_k=\dfrac{1}{2}\Delta|\n f|^2$,
and Theorem \ref{teo.boch-intro} is exactly the Bochner formula for the Laplacian
$$
\Delta|\n f|^2 = 2\lan\n f,\n(\Delta f)\ran + 2|\hess f|^2 + 2\ric(\n f,\n f).
$$
\end{remark}

\begin{remark}
Notice that the last two terms in (\ref{theo.bochner}) are in divergent form and thus their integrals vanish when the  manifold $M$ is compact. In applications we have some natural examples of $\phi$ satisfying $\sum_{k=1}^n\phi_{ikk} - c\sum_{k=1}^n\phi_{kki}=0$ for some constant $c$ (see Appendix) .   
\end{remark}

Of course, an application of the new Bochner formula is to recover the well-known Lich- nerowicz-Obata Theorem about the first eigenvalue for the Laplacian \cite{lich} and \cite{obata}.

{\bf Theorem.}
{\it Let $M$ be an $n$-dimensional compact Riemannian manifold with Ricci curvature bounded below by $(n-1)a^2$. Then the first nonzero eigenvalue $\lambda$ of the Laplacian acting on functions of $M$ satisfies $\lambda \ge na^2$  and the equality holds if and only if $M$ is isometric to the round sphere.}

Before we  state two more applications for second order differential operators, we discuss some known properties of $\square$.

Associated to tensor $\phi$ we have the $(1,1)$-tensor, still denoted by $\phi$, defined by
$$
\phi(X,Y)=\lan\phi(X),Y\ran, \forall X,Y\in TM.
$$

Here are two basic properties of the operator $\square$.

1) It follows from  Cheng and Yau (Proposition 1 in \cite{yau}) that
$$
\square f = \di(\phi(\n f)) - \sum_{i=1}^n\left(\sum_{j=1}^n\phi_{ijj}\right)f_i.
$$

2) We say that $\phi$ is \emph{divergence free} if $\di\phi\equiv0$ or, equivalently, $\displaystyle{\sum_{j=1}^n\phi_{ijj}\equiv0,}$ for all $1\leq j\leq n.$

If $M$ is compact, we know that $\square$ is self-adjoint if and only if $\phi$ is divergence free (see also \cite{yau}, Proposition $1$). If $\phi$ is symmetric and positive definite, then $\square$ is strictly elliptic. Therefore, if $\phi$ is divergence free, symmetric and positive definite, then $\square$ is a strictly elliptic and self-adjoint. Furthermore, the spectrum of $\square$ is discrete and it makes sense to consider eigenvalues, see for example \cite{Gilbarg-Trudinger}, p. $214$.

Now let us  explain the applications of Theorem \ref{teo.boch-intro} to get  estimates for the first eigenvalue for two types of operators $\square$ which arise naturally in Riemannian geometry and submanifold theory.

{\bf a)}  Let us denote by $\ric$ the \emph{Ricci tensor} of $M$. Namely 
$$
\ric(X,Y)=\sum_{i=1}^n\lan \Rm (X,e_i)Y,e_i\ran,
$$
where $\Rm(U,V)W=\n_V\n_U W - \n_U\n_V W + \n_{[U,V]}W$ is the curvature tensor of $M.$ The \emph{scalar curvature} $R$ of $M$ is defined by the trace of Ricci tensor.
We will also denote by $\ric$ the linear operator associated with the Ricci tensor, i.e.,
$
\ric(X,Y) = \lan\ric(X),Y\ran,
$
as well as its coordinates will be denoted by $\ric_{ij}.$
If $\{e_1,\ldots,e_n\}$ is an orthonormal frame, the components of curvature tensor of $M$ can be written by (see \cite{Besse}  p. 48)
\[
\begin{split}
R_{ijkl} &=\dfrac{1}{n-2}\left(\ric_{ik}g_{jl} - \ric_{il}g_{jk}+\ric_{jl}g_{ik} - \ric_{jk}g_{il}\right)\\
&- \dfrac{R}{(n-1)(n-2)}(g_{ik}g_{jl} - g_{il}g_{jk}) + W_{ijkl}.
\end{split}
\]
where $W_{ijkl}$ are the components of the \emph{Weyl tensor} $W.$ 

When $n \geq 3$, the components of  \emph{Schouten operator}  $S$ of $M$ is  defined by
$$
S_{ij}=\ric_{ij}-\frac{R}{2(n-1)}g_{ij}.
$$
In this case, one can rewrite the components of curvature tensor by
$$
R_{ijkl} =\dfrac{1}{n-2}\left(S_{ik}g_{jl} - S_{il}g_{jk} + S_{jl}g_{ik} - S_{jk}g_{il}\right) + W_{ijkl}.
$$
The operator $\square_S$ is defined by
$$
\square_S f = \sum_{i,j=1}^n S_{ij}f_{ij} = \sum_{i,j=1}^n\left(\ric_{ij} - \dfrac{R}{2(n-1)}g_{ij}\right)f_{ij}.
$$

We prove (see equation (\ref{eqn4.1}), p. \pageref{eqn4.1}) that $S$ is divergence free (or equivalently, $\square_S$ is self-adjoint) if and only if $M$ has constant scalar curvature.

\begin{definition}
A Riemannian manifold is called to have \emph{harmonic Weyl tensor} if $\di W\equiv0,$.
\end{definition}
In this case, the Schouten operator is a Codazzi operator, i.e., $S_{ijk}=S_{ikj}.$ Our first application of Theorem \ref{teo.boch-intro} is the following

\begin{theorem}\label{teo.eigen.S-intro}
Let $M^n,n\geq4$ be a compact Riemannian manifold which has harmonic Weyl tensor. If $M$ has constant scalar curvature $R$ and the Schouten tensor is positive definite, then the first nonzero eigenvalue $ \mu_1(\square_S,M) $ satisfies

\begin{equation}\label{eq.estim.S-2}
\mu_1(\square_S,M)\geq \dfrac{n-2}{2(n-1)}\left(\dfrac{R}{R-2L_0}\right)\left[L_0^2-\left(\dfrac{R}{2(n-1)}+K_0\right)L_0 + \dfrac{1}{2}K_0R\right],
\end{equation}
where
$K_0$ and $L_0$ are the lower bounds  of the sectional curvature and Ricci curvature of $M,$ respectively.
\\
Furthermore, the equality holds if and only if $M$ is the round sphere $\s^n(K_0).$
\end{theorem}

\begin{remark}
Recall that a  Riemannian manifold $(M,g)$ is said \emph{locally conformally flat} if, for any $p\in M,$ there exists a neighborhood $V$ of $p$ and a smooth function $f$ defined on $V$ such that $(V,e^{2f}g)$ is flat.
It is well known (cf. \cite{Besse}, p. 60) that $M^n,\ n\geq4,$ is locally conformally flat if and only if the Weyl tensor vanishes. In \cite{Cheng-Q-M}, Q. M. Cheng has proved that the only compact, connected oriented locally conformally flat $n-$dimensional Riemannian manifold with constant scalar curvature and non-negative Ricci curvature are those  which are quotients of  a space form or a Riemannian product $\s^1\times\s^n(\kappa).$ On the other hand, there are many examples of compact manifolds with harmonic Weyl tensor, see, for example, \cite{Besse}, p. 440-443.
\end{remark}

{\bf b)} Our second application is about isometric immersions.

Let $M^n$ be a $n$-dimensional Riemannian manifold  and  $x:M^n\ria\overline{M}^{n+1}$ be an isometric immersion of $M$ to $(n+1)$-dimensional Riemannian manifold. Denote by $A$ and $H$ the \emph{shape operator} and the mean curvature of the immersion.  
If $\lambda_1,\lambda_2,\ldots,\lambda_n$ are the eigenvalues of $A,$ i.e. the principal curvatures of the immersion, then $H=\sum_{i=1}^n\lambda_i$.

%It is well-known that when $r=1,$ $S_1=\tr A=:H$ is the \emph{non-normalized mean curvature} of $x,$ $S_n=\det A =:K$ is the \emph{Gauss-Kronecker curvature of} $x,$ and when $\overline{M}$ is the Euclidean space, we have $S_2=R,$ the scalar curvature of $M$.

The \emph{first Newton transformation} $P_1:TM\ria TM,$ associated with the second fundamental form $A,$ is defined by
$$
P_1=HI- A.
$$
Let us define the differential operator $L_1$ by
\begin{equation}\label{def.L1}
L_1f = \sum_{i,j=1}^n(P_1)_{ij}f_{ij} = \sum_{i,j=1}^n(Hg_{ij} - h_{ij})f_{ij},
\end{equation}
where $h_{ij}$ are the components of second fundamental form. Note that $P_1$ is a symmetric linear operator. The operator $L_1$ was first introduced by Voss in \cite{Voss} and appears naturally in the study of variation theory for curvature functional  $\mathcal{A}_1=\int_M H dM,$ which is called  $1-$area  {of} $M.$  See for example  \cite{Reilly} and \cite{BC} for more details.

It has been shown by Reilly, \cite{Reilly}, that $P_1$ is divergence free when $\overline{M}$ is a space form of constant sectional curvature. Therefore, under these assumptions, $L_1$ is self-adjoint.

The  eigenvalues of $L_1$ plays an important role in the study of stability for hypersurfaces with  constant $r$-mean curvature (see, for examples, \cites{MR1952173,  MR1933789, MR2084098, MR2653960, MR2048223}). In the case that $A>0,$ we have $P_1$ positive definite. Therefore, $L_1$ is an elliptic operator. We have the following first eigenvalue estimate.

\begin{theorem}\label{teo.eigen.L1-intro}
Let $x:M^n\ria\overline{M}^{n+1}(\kappa)$ be an isometric immersion of a compact Riemannian manifold into a space form of constant sectional curvature $\kappa.$ Suppose that shape operator $A$ satisfies
$$
0<\al I\leq A \leq a\al I,
$$
where $\al>0$ and $a>1$ are constants.
Then

1) when $\kappa>0$, the first nonzero eigenvalue $\mu(L_1,M)$ of operator $L_1$ satisfies
$$
\mu(L_1,M)\geq\displaystyle{\dfrac{1}{2}\left(\dfrac{na}{na-1}\right)\left[2(n-1)\al^3(n-a^2)+ 2\kappa(n-1)^2\al - \sigma\right]}
$$
where $\displaystyle{\sigma=\max_{(p,v)\in TM}\tr(\hess H|_{v^\perp})(p)}$  and $v^\perp = \{u\in T_pM;\lan u,v\ran=0\}$;

 2) when $\kappa\leq0$,  the first nonzero eigenvalue $\mu(L_1,M)$ of operator $L_1$ satisfies
$$
\mu(L_1,M)\geq\displaystyle{\dfrac{1}{2}\left(\dfrac{na}{na-1}\right)\left[2(n-1)\al^3(n-a^2)+ 2\kappa(n-1)^2a\al -\sigma \right].}
$$

Furthermore, the equalities hold if and only if $M$ is a geodesic sphere with the canonical immersion.
\end{theorem}

\begin{remark}
If $A\geq\al I>0$ then, by using Gauss equation, $$\ric \geq(n-1)[\kappa+\al^2]=\ric_{\s^n(\alpha)}>0,$$ for $\al^2>-\kappa$. That is: $\ric\geq(n-1)\Lambda>0$ for some constant $\Lambda>0$. 
Conversely, if we assume the Lichnerowicz condition $\ric\geq(n-1)\Lambda>0$, then by using Gauss equation again, we have $\lan A\circ    P_1(X),X\ran\geq(n-1)[\Lambda -\kappa]|X|^2.$ If we assume in addition that $   P_1>0$ and $\Lambda>\kappa,$ then $A$  is positive definite.  
\end{remark}

\begin{remark}
If the mean curvature  $H$ is constant and $A\geq\kappa I,$ then $x(M^n)$ is a geodesic sphere. In fact,

\begin{itemize}
\item[(1)] if $\kappa=0,$ by Hadamard theorem, cf. \cite{Hadamard}, \cite{doCarmo}, the immersion $x:M^n\ria\R^{n+1}$ is an embedding and $x(M^n)$ is a boundary of a convex domain of $\R^{n+1}.$ Thus by using the Alexandrov Theorem, cf. \cite{Alexandrov}, $x(M^n)$ is a round sphere;

\item[(2)] if $\kappa>0,$ by do Carmo-Warner Theorem, cf. \cite{doCarmo}, then $x:M^{n}\ria\s^{n+1}(\kappa)$ is an embedding and $x(M^n)$ is either totally geodesic or contained in a open hemisphere. In the last case, $x(M^n)$ is a boundary of a convex domain in $\s^{n+1}(\kappa).$ Since $A\geq\kappa I>0,$ $x(M)$ cannot be totally geodesic. Thus $x(M)$ is contained in an open hemisphere. On the other hand, in \cite{Ros}, S. Montiel and A. Ros proved that if $x:M^n\ria\s^{n+1}(\kappa)$ is an embedding such that the $r$-mean curvature $S_r$ is constant for some $r$ and $x(M^n)$ is contained in a open hemisphere, then $M^n$ is a geodesic sphere;

\item[(3)] if $\kappa<0,$ by do Carmo-Warner Theorem, cf. \cite{doCarmo}, then $x:M^n\ria\h^{n+1}(\kappa)$ is an embedding and $x(M^n)$ is a boundary of a convex domain in $\h^{n+1}(\kappa).$ On the other hand, in \cite{Ros}, S. Montiel and A. Ros proved that if $x:M^n\ria\h^{n+1}(\kappa)$ is an embedding such that $S_r$ is constant for some $r$ then $M^n$ is a geodesic sphere.

\end{itemize}
\end{remark}

The rest of the paper is organized as follows: In Section \ref{sec.2}, wel give the proof of Theorem \ref{teo.boch-intro}, in Section \ref{sec.3} we prove Theorem \ref{teo.eigen.S-intro}, and in Section \ref{sec.4} we prove Theorem \ref{teo.eigen.L1-intro}. Eventually, in the Appendix, we prove Proposition \ref{prop1.1}, which collects some properties of the Newton and Shouten tensors that we use throughout the paper.

%\begin{acknowledgements}
\emph{Acknowledgements.} The authors would like to thank  the anonymous  referee for  his/her valuable comments.
%\end{acknowledgements}

\section{A Bochner-type formula}\label{sec.2}

In this section we will prove a Bochner type formula for the differential operator $\square$ mentioned in the introduction.

\begin{proof}[Proof of Theorem \ref{teo.boch-intro}]
For a point $p\in M$,  for any orthonormal frame $\{e_1,\ldots,e_n\}$ near $p$, we have $|\nabla f|^2=\sum_{i=1}^n(f_i)^2$ and
\[
\begin{split}
\dfrac{1}{2}\square(|\n f|^2) &=\frac{1}{2}\sum_{i=1}^n \sum_{j,k=1}^n\phi_{jk}(f_i^2)_{jk} \\
&= \sum_{i=1}^n\sum_{j,k=1}^n f_i\phi_{jk}(f_i)_{jk} + \sum_{i=1}^n\sum_{j,k=1}^n\phi_{jk}(f_i)_j(f_i)_k.
\end{split}
\]

Now  we choose an orthonormal frame $\{e_1,\ldots,e_n\}$ such that $\phi$ is diagonalized  at $p$, i.e. $\phi_{jk}=\mu_jg_{jk}$, where $\mu_j$ are eigenvalues of $(\phi_{jk})$. Then we choose an orthonormal frame in a neighborhood of $p\in M$ by parallel translating  frame $\{e_1,\ldots,e_n\}$ at $p$. Here at $p$, we have $\nabla_{e_i}e_j=0$ at $p$. Moreover, $\nabla_{e_i}e_j=0$  along the geodesic tangent to $e_i$ which implies $\nabla_{e_i}\nabla_{e_i}e_j=0$ at $p$ for all $i, j$. Thus we have
\[
\begin{split}
\dfrac{1}{2}\square(|\n f|^2)
&= \sum_{i,j=1}^n f_i\mu_j(f_i)_{jj} + \sum_{i,j=1}^n\phi_{jj}(f_i)_j(f_i)_j.
\end{split}
\]
Since the terms $(f_i)_j$ and $(f_i)_{jj}$  denote differentiation of the function $f_i$,  in general they are not equal to the covariant derivatives  $f_{ij}$ and $f_{ijj}$ of $f$. However, by the special choice of our frame, we have $(f_i)_j=f_{ij}$ and $(f_i)_{jj}=f_{ijj}$ at $p$. Therefore, at $p$
\[
\begin{split}
\sum_{i,j=1}^n f_i\mu_j(f_i)_{jj} &= \displaystyle{\sum_{i,j,k=1}^nf_if_{jki}\phi_{jk} + \sum_{i,j,k,m=1}^nf_if_mR_{mjik}\phi_{jk}}\\
&=\displaystyle{\sum_{i,j,k=1}^nf_i(f_{jk}\phi_{jk})_i - \sum_{i,j,k=1}^nf_if_{jk}\phi_{jki} + \sum_{i,j,k,m=1}^nf_if_mR_{mjik}\phi_{jk}}.
\end{split}
\]
Hence
\begin{equation}\label{eq.1}
\dfrac{1}{2}\square(|\n f|^2) = \lan\n f, \n(\square f)\ran + \sum_{i,j,k=1}^n\phi_{ij}f_{jk}f_{ki} + \sum_{i,j,k,m=1}^nf_if_mR_{mjik}\phi_{jk}-\sum_{i,j,k=1}^nf_if_{jk}\phi_{jki}.
\end{equation}
On the other hand,
\[
\begin{split}
-\sum_{i,j,k=1}^nf_if_{jk}\phi_{jki}&=\sum_{i,j,k=1}^nf_if_{jk}(\phi_{jik}-\phi_{jki}) - \sum_{i,j,k=1}^nf_if_{jk}\phi_{jik}\\
&=\sum_{k=1}^n\left(\sum_{i,j=1}^nf_if_j(\phi_{jik} - \phi_{jki})\right)_k - \sum_{i,j,k=1}^nf_if_j(\phi_{jik}-\phi_{jki})_k\\
&\qquad -\sum_{i,j,k=1}^nf_{ik}f_j(\phi_{jik} - \phi_{jki}) - \sum_{i,j,k=1}^nf_if_{kj}\phi_{jik}\\
&=\sum_{k=1}^n\left(\sum_{i,j=1}^nf_if_j(\phi_{jik} - \phi_{jki})\right)_k - \sum_{i,j,k=1}^nf_if_j(\phi_{jik}-\phi_{jki})_k\\
&\qquad - \sum_{i,j,k=1}^nf_jf_{ik}\phi_{jik},
\end{split}
\]
when we used in the last equality that $\displaystyle{\sum_{i,j,k=1}^n f_jf_{ik}\phi_{jki}=\sum_{i,j,k=1}^nf_if_{kj}\phi_{jik}.}$ Then
\begin{equation}\label{eq.AB}
\begin{split}
\dfrac{1}{2}\square\left(|\n f|^2\right) &= \lan\n f, \n(\square f)\ran + \sum_{i,j,k=1}^n\phi_{ij}f_{jk}f_{ki} + \sum_{i,j,k,m=1}^nf_if_mR_{mjik}\phi_{jk}\\
&\qquad + \sum_{k=1}^n\left(\sum_{i,j=1}^nf_if_j(\phi_{jik} - \phi_{jki})\right)_k - \sum_{i,j,k=1}^nf_if_j(\phi_{jik}-\phi_{jki})_k\\
&\qquad - \sum_{i,j,k=1}^nf_jf_{ik}\phi_{jik}.\\
\end{split}
\end{equation}

In order to find a more suited expression for the term $\displaystyle{-\sum_{i,j,k=1}^nf_if_j(\phi_{jik}-\phi_{jki})_k}$ above, we use the following computation (see \cite{yau}, EQ. (2.4), P.198):
\[
\begin{split}
\Delta\phi_{ij}&=\sum_{k=1}^n(\phi_{ijkk} - \phi_{ikjk}) + \sum_{k=1}^n(\phi_{ikkj} - c\phi_{kkij}) + c\left(\sum_{k=1}^n \phi_{kk}\right)_{ij}\\
& -\sum_{m,k=1}^n\phi_{mk}R_{mikj} - \sum_{m,k=1}^n\phi_{im}R_{mkkj},
\end{split}
\]
which implies
\begin{equation}\label{eq.A}
\begin{split}
-\sum_{i,j,k=1}^nf_if_j(\phi_{ijk}-\phi_{ikj})_k &= \sum_{i,j=1}^nf_if_j\left(\sum_{k=1}^n\phi_{ikk}-c\sum_{k=1}^n\phi_{kki}\right)_j + c\sum_{i,j=1}^n(\tr\phi)_{ij}f_if_j\\
&\qquad - \sum_{i,j,k,m=1}^nf_if_j\phi_{mk}R_{mikj} + \sum_{i,j,k,m=1}^nf_if_j\phi_{im}R_{mkik}\\
&\qquad - \sum_{i,j=1}^nf_if_j\Delta\phi_{ij}.\\ 
\end{split}
\end{equation}
Rewriting the last term in right hand side of $(\ref{eq.AB})$ we have
\begin{equation}\label{eq.B}
\begin{split}
-\sum_{i,j,k=1}^nf_jf_{ik}\phi_{jik}&=-\sum_{i,j,k=1}^n(f_jf_{ik}\phi_{ji})_k + \sum_{i,j,k=1}^nf_{jk}f_{ik}\phi_{ji} + \sum_{i,j,k=1}^nf_jf_{ikk}\phi_{ij}\\
&=-\sum_{k=1}^n\left(\sum_{i,j=1}^n f_j\phi_{ij}f_{ik}\right)_k + \sum_{i,j,k=1}^n\phi_{ij}f_{jk}f_{ki}\\
&\qquad \sum_{i,j,k=1}^nf_jf_{kki}\phi_{ij} + \sum_{i,j,k,m=1}^nf_jf_mR_{mkik}\phi_{ij}.
\end{split}
\end{equation}
Equation $(\ref{theo.bochner})$ follows by replacing the expressions of $(\ref{eq.A})$ and $(\ref{eq.B})$ in $(\ref{eq.AB}),$ and noting that \[\displaystyle{\sum_{i,j,k,m=1}^n f_if_j\phi_{mk}R_{mikj}=\sum_{i,j,k,m=1}^nf_if_m\phi_{jk}R_{mjik}.}\]
\end{proof}

\section{Estimate of first eigenvalue of the Shouten operator}\label{sec.3}

%We begin this section with an algebraic Lemma.

The main purpose of this section is to prove Theorem \ref{teo.eigen.S-intro}. We start with  proving two lemmas.

\begin{lemma}[Generalized Newton inequality]\label{lem.new}
Let $A$ and $B$ be two $n\times n$ symmetric matrices. If $B$ is positive definite, then
\begin{equation}\label{eq.new}
\tr(A^2B)\geq\dfrac{[\tr(AB)]^2}{\tr B},
\end{equation}
and the equality holds if and only if $A=\al I$ for some $\al\in\R.$
\end{lemma}
\begin{proof}
Let  $C$ be a positive definite matrix. By using the Cauchy-Schwarz inequality with $A\sqrt{C}$ and $(\sqrt{C})^{-1}B,$ and the fact $\tr[(AB)^2]\leq\tr(A^2B^2),$ which holds for symmetric matrices, we have
$$
[\tr(AB)]^2=\tr(A\sqrt{C}(\sqrt{C})^{-1}B)^2\leq\tr(A^2C)\tr(B^2C^{-1}).
$$
In particular, since  $B$ is  positive definite, we can choose  $C=B$ to obtain
$$
[\tr(AB)]^2\leq\tr(A^2B)\tr B,
$$
i.e.,
%\begin{equation}\label{eq.new}
$$
\tr(A^2B)\geq\dfrac{[\tr(AB)]^2}{\tr B}.
$$
%\end{equation}
The equality holds if and only if
$$
A\sqrt{B}=\al(\sqrt{B})^{-1}B \Leftrightarrow (A\sqrt{B})\sqrt{B}=\al(\sqrt{B})^{-1}B\sqrt{B} \Leftrightarrow AB = \al B \Leftrightarrow A=\al I.
$$
%%%\qed
\end{proof}
\begin{remark} When $B=I,$  the inequality $(\ref{eq.new})$ becomes
$$
\|A\|^2\geq\dfrac{1}{n}(\tr A)^2,
$$
which is known as the \emph{(classical) Newton inequality}.
\end{remark}
When  $A=[f_{ij}]_{n\times n}$ and $B=[\phi_{ij}]_{n\times n},$ $(\ref{eq.new})$ implies
\begin{lemma}\label{lem.lich} If  $\phi$ is positive definite, then
$$
\sum_{i,j,k=1}^n\phi_{ij}f_{jk}f_{ki}\geq\dfrac{(\square f)^2}{\tr\phi},
$$
and the equality holds if and only if, $f_{ij}=\al g_{ij},$ i.e., $\hess f(X,Y) = \al\lan X,Y\ran.$
\end{lemma}

In Proposition \ref{prop1.1}, item $(4),$ in the Appendix, we have
$$
\di S = \n(\tr S).
$$
Therefore $S$ is divergence free if and only if $M$ has constant scalar curvature. Now let us prove Theorem \ref{teo.eigen.S-intro}.

\begin{proof}[Proof of Theorem \ref{teo.eigen.S-intro}.]
Since $S_{jik}=S_{jki}$ and $\di S=0,$ the Bochner formula of Theorem \ref{teo.boch-intro} in this case becomes
\begin{equation}\label{eq.boch.S}
\begin{split}
\dfrac{1}{2}\square_S|\n f|^2&=\lan\n f, \n (\square_S f)\ran + 2\sum_{i,j,k=1}^nS_{ij}f_{jk}f_{ki} + \lan S(\n f),\n(\Delta f)\ran\\
&\qquad + 2\ric(\n f,S(\n f)) -\sum_{i,j=1}^nf_if_j\Delta S_{ij}- \sum_{k=1}^n\left(\sum_{i,j=1}^nf_jS_{ij}f_{ik}\right)_k.\\
\end{split}
\end{equation}

Let us integrate and estimate each terms in the equation (\ref{eq.boch.S}). We will complete our proof after proving two claims.

\medskip
\begin{claim} Let $\mu>0$ and a smooth function $f:M\ria\R$ such that $\square_S f = -\mu f$, then
$$
\int_M \lan S(\n f), \n (\Delta f)\ran dM =-\mu\int_M|\n f|^2dM.
$$
\end{claim}
In fact, by using Divergence Theorem, we have
\[
\begin{split}
\displaystyle{\int_M \lan S(\Delta f),\n(\Delta f)\ran dM}
=& \displaystyle{\int_M \di(\Delta f\cdot S(\n f))dM -\int_M \Delta f\cdot\di(S(\n f)) dM}\\
=&\displaystyle{-\int_M \Delta f\cdot\square_S f dM}\\
=&\displaystyle{\mu\int_M f\Delta f dM}\\
=&\displaystyle{-\mu\int_M|\n f|^2dM.}\\
\end{split}												
\]
%By the above inequality, the claim follows from the definition of $a$.
%\medskip
\begin{claim}
$$
\int_M\left[ 2\ric(\n f, S(\n f)) - \sum_{i,j=1}^nf_if_j\Delta S_{ij}\right]dM \geq \Gamma\int_M|\n f|^2dM.
$$
where
\[
\Gamma = L_0^2 - \left(\dfrac{R}{2(n-1)} + K_0\right)L_0 + \dfrac{1}{2}K_0R.
\]
\end{claim}

To prove this claim, first note that
\[
\begin{split}
\ric(\n f,S(\n f)) &= \lan\ric(\n f),S(\n f)\ran\\
& = \left\lan\left(S+\dfrac{R}{2(n-1)}I\right)(\n f), S(\n f)\right\ran\\
&= |S(\n f)|^2 + \dfrac{R}{2(n-1)}\lan S(\n f),\n f\ran.
\end{split}
\]
Since $S$ is a Codazzi tensor,
$$
(\Delta S)_{ij} := \sum_{k=1}^n S_{ijkk} = \sum_{k=1}^n S_{ikjk} = \sum_{k=1}^n S_{kijk}.
$$
By using Ricci identity
$$
S_{kijk} = S_{kikj} + \sum_{m=1}^n S_{mk}R_{mijk} + \sum_{m=1}^nS_{mi}R_{mkjk},
$$
and following a computation of Cheng and Yau, cf. \cite{yau}, we have
\[
\begin{split}
(\Delta S)_{ij} &= \displaystyle{\sum_{k=1}^n S_{kkij} + \sum_{m,k=1}^n S_{mk}R_{mijk} + \sum_{m,k=1}^n S_{mi}R_{mkjk}}\\
								&=\displaystyle{(\tr S)_{ij} + \sum_{m=1}^n S_{mi}\ric_{mj} + \sum_{m,k=1}^n S_{mk}R_{mijk}}\\
								&=\displaystyle{\sum_{m=1}^n S_{mi}\left(S_{mj}+\dfrac{R}{2(n-1)}g_{mj}\right)+ \sum_{m,k=1}^n S_{mk}R_{mijk}}\\
&=\displaystyle{\sum_{m=1}^n S_{mi}S_{mj} + \dfrac{R}{2(n-1)}\sum_{m=1}^n S_{mi}g_{mj} + \sum_{m,k=1}^n S_{mk}R_{mijk}}\\
&=\displaystyle{(S^2)_{ij} + \dfrac{\tr S}{n-2}S_{ij} + \sum_{m,k=1}^n S_{mk}R_{mijk}.}
\end{split}
\]
We now choose  an orthonormal  frame such that $S_{ij}=\lambda_ig_{ij}$ at a point $p\in M.$ Let $K(u,v)$ denote the sectional curvature of the plane generated by $u,v.$ Then

\begin{eqnarray*}
&&2\ric(\n f, S(\n f))- \sum_{i,j=1}^nf_if_j\Delta S_{ij}\\
&=&\displaystyle{2\sum_{i=1}^n\left(\lambda_i + \dfrac{R}{2(n-1)}\right)\lambda_if_i^2 - \sum_{i=1}^n\lambda_i^2f_i^2 - \dfrac{R}{2(n-1)}\sum_{i=1}^n\lambda_if_i^2}\\
&&\displaystyle{-\sum_{i,j,k=1}^n\lambda_kR_{kijk}f_if_j}\\
&=&\displaystyle{\sum_{i=1}^n\lambda_i^2f_i^2 + \dfrac{R}{2(n-1)}\sum_{i=1}^n\lambda_if_i^2 + \sum_{k=1}^n\lambda_kK(e_k,\n f)[|\n f|^2-\lan e_k,\n f\ran^2]}\\
&\geq&\displaystyle{\sum_{i=1}^n\lambda_i^2f_i^2 + \dfrac{R}{2(n-1)}\sum_{i=1}^n\lambda_if_i^2 +K_0\sum_{k=1}^n\lambda_k[|\n f|^2 - \lan e_k,\n f\ran^2]}\\
&=&\displaystyle{\sum_{i=1}^n\lambda_i^2f_i^2 + \left(\dfrac{R}{2(n-1)}-K_0\right)\sum_{i=1}^n\lambda_if_i^2 + \dfrac{n-2}{2(n-1)}K_0R|\n f|^2}.
\end{eqnarray*}
Note that, if $K_0<0,$ then $\dfrac{R}{2(n-1)}-K_0>0,$ and if $K_0>0,$ then
\[
\begin{split}
\dfrac{R}{2(n-1)}-K_0 =& \dfrac{1}{2(n-1)}\sum_{i,j=1}^nK(e_i,e_j) - K_0 \\
\geq& \dfrac{n(n-1)}{2(n-1)}K_0 - K_0 \\
= &\left(\dfrac{n}{2}-1\right)K_0>0.
\end{split}
\]
It implies,
\[
\begin{split}
2\ric(\n f, S(\n f)) & - \sum_{i,j=1}^nf_if_j\Delta S_{ij}\\
\geq&\displaystyle{\sum_{i=1}^n\lambda_i^2f_i^2 + \left(\dfrac{R}{2(n-1)}-K_0\right)\sum_{i=1}^n\lambda_if_i^2 + \dfrac{n-2}{2(n-1)}K_0R|\n f|^2}\\
\geq&\displaystyle{\left[\lambda_0^2 + \left(\dfrac{R}{2(n-1)}-K_0\right)\lambda_0 + \dfrac{n-2}{2(n-1)}K_0R\right]|\n f|^2,}\\
\end{split}
\]
%
%Thus
%\begin{eqnarray}\label{est.gamma.S}
%\nonumber &&2\ric(\n f, S(\n f)) + \hess(\tr S)(\n f, \n f) - \lan(\Delta S)(\n f),\n f\ran\\
%\nonumber&\geq& 
%\displaystyle{\left[\lambda_0^2 + \left(\dfrac{R}{2(n-1)}-K_0\right)\lambda_0 + \dfrac{n-2}{2(n-1)}K_0R\right]|\n f|^2.}
%\end{eqnarray}
where $\displaystyle{\lambda_0=\min_{p\in M}\left\{\min_{1\leq i \leq n} \lambda_i(p)\right\}.}$ Since $\lambda_0=L_0-\dfrac{R}{2(n-1)},$ where $L_0$ is the minimum of the Ricci curvature, then the claim follows from the definition of $\Gamma$.

Now we are ready to complete the proof of Theorem \ref{teo.eigen.S-intro}. Since $\square_S f=-\mu f,$ we have
\begin{equation}\label{eq-S-1}
\int_M \lan \n f, \n (\square_S f)\ran = -\mu\int_M |\n f|^2 dM
\end{equation}
and, by using the Lemma \ref{lem.new},
\begin{equation}\label{eq-S-2}
2\int_M \left(\sum_{i,j,k=1}^nS_{ij}f_{jk}f_{ki}\right)dM \geq 2\int_M \dfrac{(\square_S f)^2}{\tr S}dM\geq\dfrac{2\mu^2}{\tr S}\int_M f^2 dM.
\end{equation}

Since, by using Divergence Theorem, $\displaystyle{\int_M\left[\sum_{k=1}^n\left(\sum_{i,j=1}^nf_j\phi_{ij}f_{ik}\right)_k\right]dM=0},$ then replacing these estimates in the equation ($\ref{eq.boch.S}$), p. $\pageref{eq.boch.S},$ we have
\begin{equation}\label{eq.final-1}
0\geq -2\mu\int_M|\n f|^2dM + \dfrac{2\mu^2}{\tr S}\int_M f^2 dM + \Gamma\int_M|\n f|^2 dM.\\
\end{equation}
Since
\[
\begin{split}
\displaystyle{\int_M |\n f|^2dM} &\leq\displaystyle{\dfrac{1}{\lambda_0}\int_M\lan S(\n f), \n f\ran dM}\\
&=\displaystyle{\dfrac{\mu}{\lambda_0}\int_M f^2dM,}\\
\end{split}
\]
we obtain
\[
\begin{split}
0&\geq\left(\Gamma - 2\mu\right)\int_M|\n f|^2dM + \dfrac{2\mu\lambda_0}{\tr S}\int_M |\n f|^2dM\\
&=\left(\Gamma - 2\mu\left(1-\dfrac{\lambda_0}{\tr S}\right)\right)\int_M|\n f|^2dM.
\end{split}
\]
Thus
\[
\begin{split}
\mu&\geq\dfrac{\Gamma}{2}\left(\dfrac{\tr S}{\tr S-\lambda_0}\right)\\
&=\dfrac{n-2}{2(n-1)}\left(\dfrac{R}{R-2L_0}\right)\Gamma\\
&=\dfrac{n-2}{2(n-1)}\left(\dfrac{R}{R-2L_0}\right)\left[L_0^2 - \left(\dfrac{R}{2(n-1)}+K_0\right)L_0 + \dfrac{1}{2}K_0R\right].
\end{split}
\]
To prove the equality case, we suppose $K_0=1$ and $M^n=\s^n.$ In this case $S=\dfrac{n-2}{2}I,$ $\square_S f=\dfrac{n-2}{2}\Delta f$ and $\Gamma =\dfrac{(n-2)(n-1)}{2}.$ Then the estimate becomes equality. Conversely, if the equality holds, Lemma $\ref{lem.new},$ p. $\pageref{lem.new},$ gives us that $\hess(f)=\al I.$  Following the proof of Obata Theorem step-by-step, cf. \cite{obata}, we can see that $M$ is a sphere.
\end{proof}

\section{The estimate of the first eigenvalue of $L_1$}\label{sec.4}

This section will give the proof of Theorem \ref{teo.eigen.L1-intro}.
We start with the following lemma.
\begin{lemma}\label{lem.proof-1}
Let $x:M^n\ria\overline{M}^{n+1}(\kappa)$ be an isometric immersion of a $n$-dimensional Riemannian manifold $M$ into a $(n+1)$-dimensional space form $\overline{M}$ of constant sectional curvature $\kappa.$ Then 
\begin{equation}\label{eq.proof-1}
2\ric(\n f,P_1(\n f)) - \lan(\Delta P_1)(\n f),\n f\ran = \hess(H)(\n f,\n f) - (\Delta H)|\n f|^2 + \lan Q(A)(\n f),\n f\ran,
\end{equation}
where $\displaystyle{\lan(\Delta P_1)(\n f),\n f\ran=\sum_{i,j=1}^nf_if_j\Delta(Hg_{ij}-h_{ij})}$ and $Q(A)=2A^3-3HA^2+(2H^2-|A|^2-\kappa(n-2))A + \kappa(2n-3)HI.$
\end{lemma}
\begin{proof}
%By using Simons type equation, cf. \cite{MR0266109}, \cite{Si},
By following Schoen-Simon-Yau's computations, see \cite{SSY}, eq. $(1.20),$ p. $278,$ we have
\[
\Delta h_{ij}=H_{ij} + (\kappa n - |A|^2)h_{ij} - \kappa Hg_{ij} + H\sum_{k=1}^n h_{ik}h_{kj},
\]
equivalently,
$$
(\Delta A)(X)=\n_{X}\n H + (\kappa n - |A|^2)A(X) - \kappa HX + HA^2(X).
$$
It implies
\[
\begin{split}
-\lan(\Delta P_1)(\n f),\n f\ran=&\lan(\Delta A - (\Delta H)I)(\n f),\n f\ran \\
=& \lan(\Delta A)(\n f),\n f\ran - \Delta H |\n f|^2\\
=&\hess(H)(\n f, \n f) + (\kappa n - |A|^2)\lan A(\n f),\n f\ran \\
&+ H\lan A^2(\n f), \n f\ran - \kappa H|\n f|^2 - \Delta H|\n f|^2.\\
\end{split}
\]
By using Gauss equation
\[
\begin{split}
\lan R(X,Y)Z,T\ran &=\kappa(\lan X,Z\ran\lan Y,T\ran-\lan Y,Z\ran\lan X,T\ran)\\
&\qquad+\lan A(X),Z\ran\lan A(Y),T\ran-\lan A(Y),Z\ran\lan A(X),T\ran,	
\end{split}
\]
we have
\[
\begin{split}
\lan R(\n f,e_i)P_1(\n f),e_i\ran &=\kappa(\lan \n f,P_1(\n f)\ran\lan e_i,e_i\ran-\lan e_i,P_1(\n f)\ran\lan \n f,e_i\ran)\\
									 &\qquad+\lan A(\n f),P_1(\n f)\ran\lan A(e_i),e_i\ran-\lan A(e_i),P_1(\n f)\ran\lan A(\n f),e_i\ran.
\end{split}
\]
After tracing, we obtain
\[
\begin{split}
2\ric(\n f,P_1(\n f))&=2\kappa(n-1)\lan \n f,P_1(\n f)\ran + 2H\lan A(\n f),P_1(\n f)\ran - 2\lan A^2(\n f),P_1(\n f)\ran\\
										&=2\kappa(n-1)\lan\n f,(HI-A)(\n f)\ran + 2H\lan A(\n f),(HI-A)(\n f)\ran\\
										&\qquad - 2\lan A^2(\n f),(HI-A)(\n f)\ran\\
										&=2\kappa(n-1)H|\n f|^2+(2H^2-2\kappa(n-1))\lan A(\n f),\n f\ran\\
										&\qquad- 4H\lan A^2(\n f),\n f\ran + 2\lan A^3(\n f),\n f\ran.\\
\end{split}
\]
Then
$$
2\ric(\n f,P_1(\n f)) - \lan(\Delta P_1)(\n f),\n f\ran = \hess(H)(\n f,\n f) - (\Delta H)|\n f|^2 + \lan Q(A)(\n f),\n f\ran.
$$
\end{proof}

Next lemma is a local estimate for $Q(A).$

\begin{lemma}\label{lem.proof-2}
If $0<\al I\leq A \leq a\al I$ then,
\begin{itemize}
\item[(i)] if $\kappa>0,$
\[
\lan Q(A)(X),X\ran\geq\left[2(n-1)\al^3(n-a^2)+ 2\kappa(n-1)^2\al\right]|X|^2,\\
\]
\item[(i)] if $\kappa\leq0$,
\[
\lan Q(A)(X),X\ran \geq\left[2(n-1)\al^3(n-a^2)+ 2\kappa(n-1)^2a\al\right]|X|^2,
\]
\end{itemize}
for any $X\in TM.$
\end{lemma}

\begin{proof}
Let $\{e_1,\ldots,e_n\}$ be an orthonormal base of eigenvectors of the shape operator $A,$ and $h_1,h_2,\ldots,h_n$ be its eigenvalues. Denote by $\displaystyle{S_2=\sum_{i<j}h_ih_j = \frac{1}{2}(H^2 - |A|^2)}$. For $\kappa\geq0,$ we have
\begin{equation}\label{eq.QA}
\begin{split}
\lan Q(A)(e_i),e_i\ran &=2h_i^3 - 3Hh_i^2 + (2H^2-|A|^2)h_i - \kappa(n-2)h_i + \kappa(2n-3)H\\
&=2h_i^3 - 3h_i^2(h_i+H-h_i) + 2(H^2-|A|^2)h_i + |A|^2h_i \\
&\qquad+ \kappa[(n-2)(H-h_i)+(n-1)H]\\
&=(|A|^2-h_i^2)h_i + 2 S_2h_i - 3h_i^2(H-h_i)+ \kappa[(n-2)(H-h_i)+(n-1)H]\\
&=h_i\left[(|A|^2-h_i^2) + 2 S_2 - 3h_i(H-h_i)\right]+\kappa[(n-2)(H-h_i)+(n-1)H]\\
&=h_i\left[\left(\sum_{j=1}^n h_j^2 - h_i^2\right) + 2S_2 - 3\sum_{j=1}^nh_ih_j + 3h_i^2\right]\\
&\qquad+\kappa[(n-2)(H-h_i)+(n-1)H]\\
&=h_i\left[2S_2 + 2h_i^2 + \sum_{j=1}^n h_j^2 - 2\sum_{j=1}^n h_ih_j - \sum_{j=1}^n h_ih_j\right]\\
&\qquad+\kappa[(n-2)(H-h_i)+(n-1)H]\\
&\geq\displaystyle{h_i\left[2S_2 + 2h_i^2 + \sum_{j=1}^n h_j^2 - \sum_{j=1}^n (h_i^2 + h_j^2) - \sum_{j=1}^n h_ih_j\right]}\\
&\qquad+ \kappa[(n-2)(H-h_i)+(n-1)H]\\
&= h_i\left[2S_2 - (n-1)h_i^2 - h_i(H-h_i) \right]\\
& \qquad+ \kappa[(n-2)(H-h_i)+(n-1)H]\\
&\geq2\al[S_2 - (n-1)a^2\al^2] + \kappa[(n-2)(n-1)\al+(n-1)n\al]\\
&\geq2\al[S_2 - (n-1)a^2\al^2] + 2\kappa(n-1)^2\al\\
&\geq2\al(n(n-1)\al^2 - (n-1)a^2\al^2) + 2\kappa(n-1)^2\al\\
&=2(n-1)\al^3(n-a^2)+ 2\kappa(n-1)^2\al
\end{split}
\end{equation}
and if $\kappa<0,$
\begin{equation}\label{eq.QA-2}
\begin{split}
\lan Q(A)(e_i),e_i\ran &\geq 2\al[S_2 - (n-1)a^2\al^2] + 2\kappa(n-1)^2a\al\\
&\geq 2(n-1)\al^3(n-a^2)+ 2\kappa(n-1)^2a\al.
\end{split}
\end{equation}

\end{proof}

Now we are ready to prove Theorem \ref{teo.eigen.L1-intro}.

\begin{proof}[Proof of Theorem \ref{teo.eigen.L1-intro}]
Applying the formula of Theorem \ref{teo.boch-intro} to the operator $L_1,$ and by using Codazzi Equation $h_{jik}=h_{jki},$ we have
\begin{equation}\label{for.boch.L1}
\begin{split}
\dfrac{1}{2}L_1|\n f|^2=&\lan\n f, \n (L_1 f)\ran + \lan P_1(\n f),\n(\Delta f)\ran + 2\sum_{i,j,k=1}^n(Hg_{ij}-h_{ij})f_{jk}f_{ki}\\
&\qquad + 2\ric(\n f,P_1(\n f)) -\lan(\Delta P_1)(\n f),\n f\ran- \sum_{k=1}^n\left(\sum_{i,j=1}^nf_j(Hg_{ij}-h_{ij})f_{ik}\right)_k\\
&\qquad + \sum_{k=1}^n\left(|\n f|^2H_k -\lan\n H,\n f\ran f_k\right)_k.
\end{split}
\end{equation}
Integrating this formula and, by using the divergence theorem and the fact that $L_1$ is divergence free, we have
\begin{equation}\label{for.boch-int.L1}
\begin{split}
0=&\displaystyle{\int_M\lan\n f, \n (L_1 f)\ran dM + \int_M\lan P_1(\n f),\n(\Delta f)\ran dM + 2\int_M\left(\sum_{i,j,k=1}^n(Hg_{ij}-h_{ij})f_{jk}f_{ki}\right)dM} \\
&\displaystyle{ +\int_M[2\ric(\n f,P_1(\n f))dM -\lan(\Delta P_1)(\n f),\n f\ran] dM.}\\
\end{split}
\end{equation}
Let us estimate each of these integrals. The three first integrals in expression (\ref{for.boch-int.L1}) have canonical estimates, as follows.
Since $L_1f=-\mu f$ we have
$$
\int_M \lan \n f, \n (L_1 f)\ran dM = -\mu\int_M |\n f|^2 dM.
$$
By using Divergence Theorem in the expression
\[
\begin{split}
\di(\Delta f P_1(\n f)) =& \Delta f\di(P_1(\n f)) + \lan P_1(\n f),\n(\Delta f)\ran\\
                        =& \Delta f\cdot L_1 f + \lan P_1(\n f),\n(\Delta f)\ran,\\
\end{split}
\]
we obtain
\[
\begin{split}
\displaystyle{\int_M\lan P_1(\n f),\n (\Delta f)\ran dM }=&\displaystyle{-\int_M \Delta f\cdot L_1 f dM }\\
=& \displaystyle{\mu\int_M f\Delta f dM}\\
=&\displaystyle{-\mu\int_M|\n f|^2dM.}
\end{split}
\]
Applying Lemma \ref{lem.new}, p. \pageref{lem.new}, we obtain
$$
2\int_M\left(\sum_{i,j,k=1}^n(Hg_{ij}-h_{ij})f_{jk}f_{ki}\right)dM \geq2\int_M\dfrac{(L_1f)^2}{(n-1)H}dM \geq \dfrac{2\mu^2}{n(n-1)a\al}\int_M f^2 dM,
$$
where we have used that $\tr P_1=(n-1)H.$ To estimate the last integral, we claim that, for $\kappa\geq0,$
\[
\begin{split}
\int_M[2\ric(\n f,P_1(\n f))dM &-\lan(\Delta P_1)(\n f),\n f\ran] dM\\
 &\geq \displaystyle{[2\al^3(n-1)(n-a^2) + 2\kappa\al(n-1)^2 -\sigma]\int_M|\n f|^2dM}\\
\end{split}
\]
and for $\kappa<0,$
\[
\begin{split}
\int_M[2\ric(\n f,P_1(\n f))dM &-\lan(\Delta P_1)(\n f),\n f\ran] dM\\
 &\geq \displaystyle{[2\al^3(n-1)(n-a^2) + 2\kappa a\al(n-1)^2-\sigma]\int_M|\n f|^2dM},\\
\end{split}
\]
where $\displaystyle{\sigma=\max_{(p,v)\in TM}\left(\tr(\hess H)|_{v^\perp}(p)\right)},$ $v^\perp=\{u\in T_pM;\lan u,v\ran=0\}.$
In fact, by using Lemma \ref{lem.proof-1}, we have
\[
\begin{split}
\int_M\left[2\ric(\n f,P_1(\n f)) - \lan(\Delta P_1)(\n f),\n f\ran\right]dM&= \int_M\left[\hess(H)(\n f,\n f) - (\Delta H)|\n f|^2\right]dM\\
& + \int_M\lan Q(A)(\n f),\n f\ran dM.\\
\end{split}
\]
By using Lemma \ref{lem.proof-2}, we have
$$
\int_M \lan Q(A)(\n f), \n f\ran dM\geq[2(n-1)\al^3(n-a^2)+ 2\kappa(n-1)^2\al] \int_M|\n f|^2dM
$$
for $\kappa>0,$ and
$$
\int_M \lan Q(A)(\n f), \n f\ran dM\geq[2(n-1)\al^3(n-a^2)+ 2\kappa(n-1)^2a\al] \int_M|\n f|^2dM
$$
for $\kappa\leq0.$
On the other hand,
\[
\begin{split}
\displaystyle{\int_M(\hess(H)(\n f,\n f) - (\Delta H)|\n f|^2)dM} =&\displaystyle{\int_M\left[\hess(H)\left(\dfrac{\n f}{|\n f|},\dfrac{\n f}{|\n f|}\right) - (\Delta H)\right]|\n f|^2dM}\\
\geq&\displaystyle{-\sigma\int_M|\n f|^2dM.}\\
\end{split}
\]
Replacing these estimates in expression $(\ref{for.boch-int.L1}),$ we obtain, for $\kappa>0,$
\[
\begin{split}
0\geq&\displaystyle{-2\mu\int_M|\n f|^2 dM - \sigma\int_M|\n f|^2 dM+\dfrac{2\mu^2}{n(n-1)a\al}\int_M f^2 dM}\\
&\displaystyle{+ [2(n-1)\al^3(n-a^2)+ 2\kappa(n-1)^2\al]\int_M|\n f|^2 dM},
\end{split}
\]
and an analogous expression for $\kappa\leq0.$
Note that
$$
(n-1)\al|\n f|^2\leq\lan P_1(\n f),\n f\ran\leq(n-1)a\al|\n f|^2.
$$
Since $\frac{1}{2}L_1(f^2)=fL_1f + \lan\n f, P_1(\n f)\ran,$ by using Divergence Theorem, we have
$$
\int_M \lan\n f, P_1(\n f)\ran dM= -\int_M fL_1f dM = \mu\int_M f^2dM.
$$
It implies
$$
\int_M f^2 dM \geq \dfrac{(n-1)\al}{\mu}\int_M|\n f|^2 dM.
$$
Denoting by $C=2(n-1)\al^3(n-a^2)+ 2\kappa(n-1)^2\al$, we have
$$
0\geq (-2\mu-\sigma+C)\int_M|\n f|^2dM + \dfrac{2\mu}{na}\int_M|\n f|^2dM,
$$
i.e.,
$$
-2\mu + \dfrac{2\mu}{na} - \sigma + C\leq0.
$$
Therefore,
$$
\mu\geq\displaystyle{\dfrac{1}{2}\left(\dfrac{na}{na-1}\right)\left[2(n-1)\al^3(n-a^2)+ 2\kappa(n-1)^2\al - \sigma\right],}
$$
for $\kappa>0,$ and
$$
\mu\geq\displaystyle{\dfrac{1}{2}\left(\dfrac{na}{na-1}\right)\left[2(n-1)\al^3(n-a^2)+ 2\kappa(n-1)^2a\al - \sigma\right],}
$$
for $\kappa\leq0.$
Now, consider the case of the canonical immersion of a geodesic sphere $x:\s^n(\alpha)\ria\overline{M}^{n+1}(\kappa).$ In this case we have $A=\alpha I,$ $a=1$ and $L_1 f = n(n-1)\al\Delta f.$ Since $\mu(L_1,M)=n(n-1)\al[\al^2+\kappa]$ then, replacing these data in the estimate, the inequality becomes equality and the estimate is sharp. On the other hand, if the equality holds, the equality case of Lemma $\ref{lem.new},$ p. $\pageref{lem.new}$ implies that $f_{ij} = cg_{ij},$ for some real constant $c,$ and following the proof of Obata Theorem, cf. \cite{obata}, we can conclude that $M$ is a geodesic sphere.
\end{proof}

\section{Appendix}
In this appendix we include the Proposition mentioned in the introduction which gives examples of tensor $\phi$, we refer to \cite{CX} for more related discussions.

\begin{proposition}\label{prop1.1}

Let $M$ be a Riemannian manifold.

\begin{itemize}
\item[(1)] If $M$ has constant scalar curvature and $c$ is a real constant, then the linear operator $S_c:=\ric-cI$ satisfies $\di S_c\equiv0$;

\item[(2)] The \emph{Einstein operator} $E:=\dfrac{1}{2}RI-\ric$ satisfies $\di E=0;$

\item[(3)] If $M$ is an immersed hypersurface in an Einstein manifold, then the Newton transformation $P_1$ satisfies $\di P_1\equiv0;$

\item[(4)] If $M$ is an immersed hypersurface in an space form of constant sectional curvature, then the Newton transformation $P_r$ satisfies $\di P_r\equiv0;$

\item[(5)] The Shouten operator $S$ satisfies $\di S=\n(\tr S);$

\item[(6)] If $M$ is locally conformally flat, then the Newton transformations $T_k(S)=T_k$ associated with $S$, $1\leq k\leq n,$ satisfies $\di T_k(S)\equiv0$.
\end{itemize}
\end{proposition}

\begin{proof}

It is well known, cf. \cite{Petersen}, p. $39$, and \cite{yau}, p. $197,$ that
$$
\di (\ric) = \dfrac{1}{2}dR.
$$
If $M$ has constant scalar curvature, then $\di(\ric)=0$, which implies that $\di(S_c)=0,$ since $c$ is constant. Claim (2) follows from
$$
\di E= \di(\ric) - \dfrac{1}{2}\di(RI)=\dfrac{1}{2}dR - \dfrac{1}{2}dR=0.
$$
The proof of claim (3) is simple and follows from well known identity
$$
\di(A)=dH,
$$
which holds for hypersurfaces immersed in an Einstein manifold, (see  \cite{MR1936088}, for a proof). We have
$$
\di (P_1) = \di (HI) - \di (A) = dH-\di(A)=0.
$$
The proof of claim (4) can be found in \cite{Reilly} or \cite{Rosenberg}. To prove claim (5), we can use the identity $\displaystyle{\sum_{j=1}^n\ric_{ijj}=\dfrac{1}{2}R_i},$ to obtain
\begin{equation}\label{eqn4.1}
\begin{split}
\sum_{j=1}^n S_{ijj}&=\sum_{j=1}^n\left(\ric_{ij}-\dfrac{R}{2(n-1)}g_{ij}\right)_j \\
&= \sum_{j=1}^n\ric_{ijj} -\sum_{j=1}^n\dfrac{R_j}{2(n-1)}g_{ij}\\
&=\dfrac{1}{2}R_i - \dfrac{R_i}{2(n-1)}\\
& = \dfrac{n-2}{2(n-1)}R_i \\
&= (\tr S)_i,
\end{split}
\end{equation}
i.e.,
$$
\di S = \n(\tr S).
$$
Claim (6) was proved by Viaclovsky, and can be found in \cite{MR1738176}.
%%%\qed
\end{proof}

\begin{tabular}{lccl}
						Hil\'ario Alencar&\ \ \ \ \ \ \ \ \ \ \ \ \ \ \ \ \ \ \ && Greg\'orio Silva Neto\\
						Instituto de Matem\'atica&&&Instituto de Matem\'atica\\
						Universidade Federal de Alagoas&&& Universidade Federal de Alagoas\\
						57072-900, Macei\'o, Brazil&&&57072-900, Macei\'o, Brazil\\
	 					hilario@mat.ufal.br&&&gregorio@im.ufal.br\\
						&&&\\
						Detang Zhou&&&\\
						Instituto de Matem\'atica e Estat\'\i stica&&&\\
						Universidade Federal Fluminense&&&\\
						24020-140, Niter\'oi, Brazil&&&\\
						zhou@impa.br&&&\\
\end{tabular}

\end{document}